\documentclass[twoside, 10pt]{article}
\usepackage{amsmath,amsthm,amssymb,amscd,amstext,amsfonts}
\usepackage[utf8]{inputenc}
\usepackage{mathtools}
\usepackage{mathrsfs}
\usepackage{thmtools}
\usepackage{latexsym}
\usepackage{verbatim}
\usepackage{framed}
\usepackage{graphicx}
\usepackage{stmaryrd}
\usepackage{fullpage}
\usepackage{bm}
\usepackage{url}
\usepackage{physics}
\usepackage{dsfont}
\usepackage{tikz}
\usepackage{titling}
\usepackage{color}
\usepackage{xpatch}
\usepackage{algorithm}
\usepackage{algorithmic}
\usepackage{epsf}
\usepackage{systeme}
\usepackage[explicit, small]{titlesec}
\titlelabel{\thetitle.\enspace}
\usepackage[small, labelfont=bf, labelsep=period]{caption}

\usepackage[breaklinks=false, linktocpage=true,]{hyperref}

\hypersetup{
	colorlinks = true,
    linkcolor={purple},
    urlcolor={purple},
    citecolor={blue!80!black}
}
\usepackage[nameinlink, noabbrev, capitalize]{cleveref}
\makeatletter
\def\tagform@#1{\maketag@@@{(\ignorespaces{\oldstylenums{#1}}\unskip\@@italiccorr)}}
\renewcommand{\eqref}[1]{\textup{{\normalfont(\oldstylenums{\ref{#1}}}\normalfont)}}

\newcommand{\ps@bookheader}{%
\renewcommand\@oddfoot{\hfil}%
\renewcommand\@evenfoot{\hfil}%
\renewcommand\@oddhead{\ifnum\value{page}>1
{\small\hfil GAVALAKIS, GOH, AND KONTOYIANNIS}\hfil\thepage\else\hfil\fi}}
\renewcommand\@evenhead{\thepage\hfil\small ENTROPY LOWER BOUNDS AND SUM-PRODUCT PHENOMENA\hfil}%
\let\over\@@over
\let\atop\@@atop
\makeatother

\usepackage{enumitem}
\setenumerate{label={\rm\roman*)}}

\titleformat{\paragraph}[runin]{\normalfont\normalsize\bfseries}{\theparagraph}{1em}{#1}[.]
\titlespacing{\paragraph}{0pt}{3.25ex plus 1ex minus .2ex}{0.7em}

\usepackage[letterpaper, total={5.5in, 8in}, vmarginratio=1:1, hmarginratio=1:1, headsep=21pt]{geometry}

\newcommand\eps{\epsilon}

\newcommand\FF{{\mathbf{F}}}
\newcommand\ex{{\mathbf{E}}}
\DeclareMathOperator{\Eta}{\mathbf{H}}

\DeclareMathOperator{\Etamin}{\mathbf{H}_\infty}
\DeclareMathOperator{\pr}{\mathbf{P}}

\newcommand\given{\mathbin{|}}
\newcommand{\one}{\mathop{\mathbf{1}}\nolimits}

\DeclareMathOperator{\dR}{d_{\mathrm{R}}}

\newcommand{\RR}{\mathbf{R}}   
\newcommand{\CC}{\mathbf{C}}   
\newcommand{\ZZ}{\mathbf{Z}}   

\declaretheoremstyle[bodyfont=\normalfont\slshape, notefont=\normalfont\itshape,notebraces={{\rm(}}{{\rm)}}, postheadspace=0.5em,headpunct={\rm.}, spaceabove=8pt, spacebelow=8pt]{slbody}

\declaretheorem[name=Theorem, numberwithin=section, style=slbody]{theorem}
\declaretheorem[name=Lemma, numberwithin=section, sibling=theorem, style=slbody]{lemma}
\declaretheorem[name=Corollary, numberwithin=section, sibling=theorem, style=slbody]{corollary}
\declaretheorem[name=Proposition, numberwithin=section, sibling=theorem, style=slbody]{proposition}
\declaretheorem[name=Conjecture, numberwithin=section, sibling=theorem, style=slbody]{conjecture}



\usepackage[textwidth=3cm]{todonotes}
\renewcommand{\maketitle}{%
  \begin{center}
    {\large\bf Entropy lower bounds and sum-product phenomena}\\
    \vskip 36pt
    {\sc Lampros Gavalakis, Marcel K.~Goh, {\rm and} Ioannis Kontoyiannis}
    \medskip
    \vskip 36pt
  \end{center}
}
\title{}\author{}\date{}

\begin{document}

\maketitle
\renewenvironment{abstract}{\quotation\noindent\small{\bfseries\abstractname.\enspace}}{\endquotation}

\begin{abstract}
Various lower bounds are established 
for the entropy of sums, products
and their combinations.
First, we derive a prime-field analogue of a version
of the entropy power  
inequality established by Tao over torsion-free groups. 
Next, we prove an entropy sum-product statement:
For independent and identically distributed 
random variables $X,X'$,
the maximum of $\Eta(X+X')$ and $\Eta(XX')$ 
is bounded below by a linear combination
of the entropy and the 
min-entropy (R\'enyi entropy of order~$\infty$)
of~$X$. This result, obtained by bounding 
entropies of the form $\Eta\bigl( X(Y+Z)\bigr)$ from above and below, 
is valid over arbitrary fields $F$. Over $F=\RR$, a slightly
stronger inequality is derived.
Finally, a weak version of a purely Shannon-entropic
sum-product result is developed:
If the entropic additive doubling of a random 
variable $X$ over an arbitrary field is $O(1)$,
then its multiplicative doubling is at least 
proportional to $\Eta(X)$.
\end{abstract}

\vskip48pt

\section{Introduction}

There has been much interest of late in entropic analogues of 
additive-combinatorial inequalities,
spurred in part by the recent proof of Marton's conjecture in finite 
groups~\cite{pfr,martontorsion} and by recent progress towards 
the polynomial Freiman-Ruzsa conjecture in torsion-free 
groups~\cite{pfrimproved}. However, the idea 
of a correspondence between 
information-theoretic bounds for the entropy
and additive combinatorial inequalities 
dates back much further, see, e.g.,~\cite{ruzsa2009,tao2010,KM:14}.

A central object of study in additive combinatorics is the 
cardinality of {\em sumsets},
$$A+B = \{a + b \in G : a\in A, b\in B\},$$
where $A,B$ are finite subsets
(which may satisfy suitable conditions)
of an abelian group $G$.
If $G = F$ is a field, then one might also 
study the size of \textit{product sets},
$$AB = \{ a\cdot b\in F : a\in A, b\in B\}.$$
The \textit{(Shannon) entropy} of a 
discrete random variable $X$ taking values in a set $A$
is
$$\Eta(X) = -\sum_{x\in A} \pr(X = x)\log\pr(X = x),$$
where, throughout this paper, $\log$ denotes the binary logarithm,
and with the usual convention that $0\log 0 =0$.
The entropy $\Eta(X)$ naturally arises as the normalized logarithm 
of the volume of the {\em typical set} of outcomes of independent
copies $X_1,\ldots,X_n$ of $X$~\cite{cover:book2},
and thus estimates for the volumes of sumsets often have 
analogues for entropies of sums. For independent random 
variables $X$ and $Y$ taking values in an abelian group, the 
entropic version of 
the first problem mentioned above is the study
of how large the 
entropy $\Eta(X+Y)$ can be; likewise, when $X$ and $Y$ take values 
in a field we also consider the entropy $\Eta(XY)$.
An important special case is when $X$ and $Y = X'$ 
are identically distributed. There, it is natural to consider the
\textit{entropic additive doubling},
$\Eta(X+X') - \Eta(X)$, and 
the \textit{entropic multiplicative doubling},
$\Eta(XX') - \Eta(X)$.

An elementary lower bound for the size of a sumset $|A+B|$ 
when $A,B\subseteq \FF_p$ is the Cauchy--Davenport 
inequality~\cite{cauchy,davenport}
\begin{equation}
|A+B| \ge \min\{ p, |A|+|B|-1\}.
\label{eq:CD}
\end{equation}
In the spirit of the log-cardinality to entropy correspondence, 
it is natural to ask whether an entropic analogue of 
Cauchy--Davenport exists. A na\"ive translation 
of~\eqref{eq:CD}
suggests the bound
$$\Eta(X+Y) \ge
\min\Bigl\{\log p, \log\bigl(2^{\Eta(X)} + 2^{\Eta(Y)} - 1\bigr)\Bigr\}.$$
However, this does not hold in general, as the following example shows: 
Let $X$ be uniform on $\{0,1\}$ and let $X_n$ be the sum of $n$ 
independent copies of $X$, so that $X_n$ has
$\mathrm{Binomial}(n,1/2)$ distribution. 
Then, by Proposition~3.1 of~\cite{GKclt},
$$\Eta(X_n) = \frac{1}{2}\log{\Big(\frac{1}{2}\pi e n\Bigr)}  
+ o_{n\to\infty}(1),$$
as long as $p$ is large enough depending on $n$ so that torsion 
does not come into play.
Since $X_{2n}$ is distributed as $X_n + X_n'$, for independent copies $X_n$ and $X_n'$ of $X_n$, we have
$$\Eta(X_n + X_n') = \Eta(X_n) + {1\over 2} + o_{n\to\infty}(1).$$
This is essentially a torsion-free construction, as also noted 
in~\cite{tao2010} for torsion-free groups. It appears to be 
a fundamental obstacle when trying to obtain entropic analogues 
of lower bounds for sumsets in (essentially) torsion-free settings 
due to the existence of discrete Gaussian examples, which do 
not have natural analogues for sets. 

In view of the above comments, a more appropriate 
candidate for an entropic Cauchy--Davenport theorem
might be that
$$\Eta(X+Y) \ge
\min\left\{\log p, \log \left({2^{\Eta(X)} + 2^{\Eta(Y)} 
- 1\over \sqrt 2}\right)
\right\}$$
for independent $X,Y$ taking values in $\FF_p$ for a prime $p$.
In particular, when $Y = X'$ is an independent copy of $X$ 
this reduces to
\begin{equation}\begin{aligned}\label{eqiidCD}
\Eta(X+X') &\ge \min
\left\{\log p, \log \left({2^{\Eta(X)+1} - 1\over \sqrt 2}\right)\right\}\cr
&= \min\left\{\log p, H(X) + \frac{1}{2} - o_{p\to\infty}(1)\right\}.
\end{aligned}\end{equation}

Our first result, proved in Section~2, is a weaker form of \eqref{eqiidCD},
which asserts (under a mild quantitative assumption that forbids, e.g., 
deterministic random variables and random variables uniformly distributed
on the whole of $F$) that the entropic additive doubling of an $X$ taking 
values in a prime field $\FF_p$ must be at least $1/2-\eps$.

\begin{theorem}\label{thmFpEPI}
For any $0<\eps<1/2$ there exist $k_\eps, K_\eps \in \RR$ 
such that, for any prime $p$ with $\log p > k_\eps + K_\eps$ and 
any $\FF_p$-valued
random variable $X$ satisfying
$$k_\eps < \Eta(X) < \log p - K_\eps,$$
we have
$$\Eta(X + X') \ge \Eta(X) + {1\over 2} - \eps,$$
where $X'$ is an independent copy of $X$.
\end{theorem}

This theorem should be compared with the statement that 
any $A\subseteq \FF_p$ satisfies $|A+A| \ge \min\bigl\{ p, 2|A|-1\bigr\}$,
by the Cauchy--Davenport inequality.
It can also be viewed as a prime-field analogue of the following
entropy-power-type inequality due to Tao over torsion-free 
groups. In fact, our proof uses Tao's result.

\begin{theorem}[{\rm\cite[Theorem~1.9]{tao2010}}\/]\label{thmtaoEPI}
For every $\eps>0$ there exists $k_\eps$ such that,
for any random variable $X$ taking values in a torsion-free group 
with $\Eta(X)>k_\eps$, we have
$$\Eta(X+X') \ge \Eta(X) + {1\over 2} - \eps,$$
where $X'$ is an independent copy of $X$.
\end{theorem}

In Section~3, we first obtain a lower 
bound on entropies of the form $H(X(Y+Z))$, 
where $X,Y,Z$ are independent random variables taking 
values in some field $F$. 
This lower bound, given in \Cref{lemHminlowerbound},
also involves
the \textit{min-entropy}, also known as the \textit{R{\'e}nyi entropy 
of order~$\infty$}, of a discrete random variable $X$ taking values
in a set $A\subseteq F$, defined as:
$$\Etamin(X) = \min_{x\in A} \log\biggl( \frac{1}{\pr(X = x)}\biggr) = 
-\log \max_{x\in A} \pr(X=x).$$
Then, combining \Cref{lemHminlowerbound} with a recent upper 
bound for $\Eta\bigl(X(Y+Z))$ by M\'ath\'e and O'Regan~\cite{matheoregan}, 
we obtain the following entropic sum-product statement.

\begin{theorem}\label{thmminentropy}
There exist absolute constants $K_1$ and $K_2$ such that the following holds. Let $X$ be a finitely supported
random variable taking values in a field $F$. If $F$ has characteristic 
$p > 0$, assume further that $\Etamin(X) \leq (2/3)\log p - K_1 $. Then
$$ \max\left\{ \Eta(X+X'), \Eta( XX') \right\} \ge 
\frac{4 \Eta(X) + 3\Etamin(X)}{6} - K_2,$$
where $X'$ is an independent copy of $X$.
\end{theorem}

In the case $F=\RR$, a slightly better bound can be obtained.

\begin{theorem}\label{thmminentropyreal}
There exists an absolute constant $K_3$ such that the following holds. 
If $X$ is a finitely supported real-valued random variable with 
$\pr(X=0) = 0$, then:
$$\max\left\{ \Eta(X+X'), \Eta(XX')\right\}
\ge \frac{4\Eta(X) + 2\Etamin(X)}{5} - K_3 \log\Etamin(X) .$$
\end{theorem}

Since $\Eta(X) = \Etamin(X) = \log|A|$ when $X$ is uniformly
distributed on $A$, Theorems~\ref{thmminentropy} and~\ref{thmminentropyreal}
immediately give an exponent of $6/5$ for the classical 
Erd\H{o}s--Szemer\'edi problem in $\RR$ and an exponent 
of $7/6$ for the corresponding problem over all other fields.

\begin{corollary}\label{corclassicalsumproduct}
Let $K_2$ and $K_3$ be the absolute constants from Theorems~\ref{thmminentropy} and~\ref{thmminentropyreal}, respectively.
Let $A$ be a finite subset of a field $F$. If $F$ has characteristic $p>0$, assume that $|A|\ll p^{2/3}$. Then
$$\max\left\{ |A+A|, |A\cdot A|\right\}\ge \frac{|A|^{7/6}}{2^{K_2}}.$$
If $F = \RR$, this bound can be improved to
$$\max\left\{ |A+A|, |A\cdot A|\right\} \ge \frac{|A|^{6/5}}{(\log|A|)^{K_3}}.$$
\end{corollary}

\begin{proof}
Let $X$ and $X'$ be independent, uniformly distributed 
random variables on $A$, so that
$\Eta(X) = \Etamin(X) = \log|A|.$
Since $X+X'$ takes values in $A+A$ and $XX'$ takes values in $A\cdot A$, 
by \Cref{thmminentropy}
we have
$$\log\max\left\{ |A+A|, |A\cdot A|\right\} \ge \frac76 \log |A|.$$
The improvement to the exponent in $\RR$ follows similarly by removing 
$0$ from $A$ (which decreases $|A|$ but to a negligible degree) 
and then invoking \Cref{thmminentropyreal} instead.
\end{proof}

For most common choices of $F$, there are better exponents known than those furnished by \Cref{corclassicalsumproduct}.
Ignoring logarithmic factors, the current best exponent for $\RR$, 
namely $4/3 + 2/951$, is due to Bloom~\cite{bloom2025}, 
and the best exponent for $\FF_p$ is $5/4$, due to Mohammadi and Stevens. 
When $F = \FF_q$ with $q=p^n$ for some prime $p$ and large $n$, 
the condition $|A|\ll p^\delta$ becomes very restrictive. 
In this setting one seeks a statement that allows for larger 
sets $A$ while still avoiding subfields of $F$;
then the best known exponent is $12/11$, due to Li 
and Roche-Newton~\cite{LiRocheNewton}.
And for the skew field of quaternions,
Basit and Lund~\cite{basitlund2019}
obtain the exponent $4/3 + c$ for 
some unspecified $c>0$,
which in turn implies the same exponent for the complex field $\CC$.

On the other hand, the appearance of min-entropy in the 
lower bound is somewhat unsatisfying. Indeed, it is conjectured in~\cite{goh2024} that the min-entropy terms 
above may be replaced by Shannon entropy, possibly to 
the detriment of the coefficient $7/6$. 
In particular,~\cite[Conjecture~18]{goh2024} states that there is a
$\delta>0$ such that, for any finitely supported
random variable $X$ with values in $\RR$,
\begin{equation}
\max\left\{\Eta(X+X'), \Eta(XX') \right\} \ge 
\bigl(1+\delta-o(1)\bigr) \Eta(X),
\label{eq:conjecture}
\end{equation}
where $X'$ is an independent copy of $X$ and
$o(1)$ denotes a term going to $0$ as $\Eta(X)$ goes to infinity.
Although this conjecture remains open, it
was noted in~\cite{lgk2025} that, if~\eqref{eq:conjecture}
is true, one must have $\delta \le 1/3$, which is strictly smaller than the corresponding exponent obtained by Bloom in the combinatorial sum-product problem. This naturally leads to the question of whether a sum-product phenomenon for entropy exists and, if so, what form it takes. 
 
Our next result, proved in Section~4, establishes weaker
versions of the conjectures in~\cite{goh2024}.
It states that, if the entropic additive doubling of a random variable is 
\textit{constant}, then its entropic multiplicative doubling must 
be substantial.
The proof uses \Cref{thmminentropy} to derive a sum-product inequality 
involving only Shannon entropy, since the constant doubling assumption 
implies sufficient structure to compare Shannon entropy with min-entropy. 

\begin{theorem}\label{thmweaksumproduct}
Let $F$ be a field, let $X$ be a random variable taking values 
in a finite subset of $F$, let $X_1$ and $X_2$ be independent copies of $X$,
and let $C,\epsilon>0$. 
There exists a constant $K_{C,\eps}$ depending
on $C$ and $\epsilon$ such that
the following holds. 
If
$$\Eta(X_1+X_2) \leq \Eta(X) + C,$$
and also assuming that 
$\Eta(X) \le (2/3)\log p - K_{C,\eps}$ if 
$F$ has positive characteristic $p$, 
then
$$\Eta(X_1 X_2) \geq \biggl(\frac{7}{6}-\epsilon \biggr)\Eta(X),$$
provided that $\Eta(X)$ is large enough depending on $C$ and $\epsilon$.
\end{theorem}

The coefficient $7/6$ here is an artefact of the coefficients 
of \Cref{thmminentropy}; in view of \Cref{thmminentropyreal}, the 
coefficient can be improved to $6/5$ in the case $F=\RR$.

\paragraph{Definitions and notation}
As mentioned, all logarithms and entropies are taken to base~2; 
up to trivial rescaling, all our results remain valid for any choice 
of base. We write $X = O(Y)$ or $X\ll Y$ to mean that $X \le CY$ for 
some absolute constant $C$, in which case we also 
equivalently write $Y\gg X$ or $Y = \Omega(X)$. 
If $C$ is not an absolute constant but instead depends on some parameters, 
we include these parameters in subscript, writing, e.g., $X = O_Z(Y)$. 
The notation $o_{n\to\infty}(1)$ indicates a quantity that 
goes to $0$ as the parameter $n$ goes to infinity, although
the subscript may be omitted when it is clear from context.

A {\it coset progression} in an additive group $G$ is a sumset $H+P$ where $H$ is
a finite subgroup of $G$ and $P$ is a generalised arithmetic progression, i.e., a set of
the form
$$ P = \bigl\{a + n_1 r_1 + \cdots + n_d r_d : n_1 \in [0, N_1) ,\ldots, n_d\in [0,N_d)\bigr\},$$
where the integer $d\ge 0$ is called the {\it rank} of the progression, $a,r_1,\ldots,r_d$
lie in $G$, and $N_1,\ldots,N_d$ are integers. The coset progression is said to be
{\it $t$-proper} for some $t>0$ if the sums $h+a+n_1r_1 + \cdots + n_dr_d$ for $h\in H$
and $n_i\in [0,tN_i)$ are distinct, and {\it proper} if it is $1$-proper.

Similarly,
a {\it symmetric} generalised arithmetic progression
is a set of the form
$$ P = \bigl\{n_1 r_1 + \cdots + n_d r_d : n_1 \in [-N_1,N_1] 
,\ldots, n_d\in [-N_d, N_d]\bigr\},$$
and a symmetric coset progression is $H+P$ where $H$ is 
a finite subgroup and $P$
is a symmetric generalised arithmetic progression.

Let $A$ be a subset of an abelian group $G$ and $B$ a subset of an abelian group $H$. A map
$\phi : A\to B$ is called a {\it Freiman isomorphism} if for all $a_1, a_2, a_1', a_2'\in A$,
we have $a_1 + a_2 = a_1' + a_2'$ if and only if $\phi(a_1) + \phi(a_2) = \phi(a_1') + \phi(a_2')$.
Note that if $A$ is Freiman isomorphic to $B$, then $|A+A| = |B+B|$.

Let $X$ be a random variable supported on a finite set $A$.
Besides ordinary Shannon entropy and min-entropy, we will
also use the \textit{collision entropy}, also know as the 
\textit{R\'enyi entropy of order $2$}, which is given by
$$\Eta_2(X) = \log \Biggl( \frac{1}{\sum_{x\in A}\pr(X = x)^2}\Biggr) 
= -\log \pr(X = X'),$$
where $X'$ is an independent copy of $X$.
It is easy to see that $\Etamin(X) \le \Eta_2(X) \le \Eta(X)$
always, and that, 
when $X$ is uniform on a set $A$, 
$\Etamin(X)= \Eta_2(X) = \Eta(X)=\log|A|$.
Given a non-null event $E$, we write $\Eta(X|E)$ for the
entropy of a random variable $Z$ having the distribution
of $X$ conditioned on $E$. And given 
another random variable $Y$, jointly distributed
with $X$, and supported on a finite set $B$,
the {\em conditional entropy}
of $X$ given $Y$ is defined by
$$\Eta(X|Y)=\sum_y \pr(Y=y)\Eta(X|Y=y).$$

\paragraph{Note added during revision}
Following the original announcement of this paper, a preprint of Li pointed out that Lemmas~\ref{lemMO} and~\ref{lempohoata} were, strictly speaking, incorrect as stated~\cite[Remark 6.3]{li2026}. Although these lemmas were used in the proofs of Theorems~\ref{thmminentropy} and~\ref{thmminentropyreal}, the theorems still remain true exactly as originally stated since the error terms that must be accounted for in the corrected versions of Lemmas~\ref{lemMO} and~\ref{lempohoata} can naturally be absorbed into the error constants $K_2$ and $K_3$. The present revision to this paper includes the corrected versions of these four results.
Li~\cite{li2026} also proves Conjecture~18 of~\cite{goh2024} described earlier, with $\delta=1/7$. However, we have left the discussion above unchanged to preserve the thread of the literature.

\section{An entropy power inequality over prime fields}

This section is devoted to the proof Theorem~\ref{thmFpEPI}, 
under the assumption that $X$ is suitably far, in an entropic sense, 
from either being deterministic or being uniform on the whole field. 
Much of the proof and the lemmas the precede it are 
adaptations of parts of Tao's proof of~\cite[Theorem~1.9]{tao2010},

\begin{lemma}\label{lemZsupport}
For any $C>0$ there is $\delta_0 > 0$ such that the following holds.
Let $G$ be an abelian group, let
$X$ and $Z$ be $G$-valued random variables with $\Eta(Z) \leq C$, and let
$X_1$ and $X_2$ be independent copies of $X$. Define
$$A = \bigl\{z\in G : \pr(Z=z)\ge \delta\bigr\},$$
and let $X'$ have the same distribution as $X$ conditioned
on $\{Z\in A\}$.
Then the entropic doubling of $X$ satisfies,
for all $0<\delta<\delta_0$,
\begin{equation}
\Eta(X_1 + X_2) - \Eta(X)
\ge \bigl(\Eta(X_1' + X_2') - \Eta(X')\bigr)
\biggl(1-\frac{C}{\log(1/\delta)}\biggr)^2 
- C\frac{\log\log(1/\delta)}{\log(1/\delta)},
\label{eq:firstclaim}
\end{equation}
where $X_1,X_2$ are independent copies of $X$
and $X_1',X_2'$ are independent copies of $X'$.
If $G=F$ is also a field, then 
$$\frac{\Eta(X_1 X_2)}{\pr(X\ne 0)} - \Eta(X)
\ge \bigl(\Eta(X_1' X_2') - \Eta(X')\bigr)
\frac{\pr(Z\in A)^2}{\pr(X\ne 0)} 
- C\frac{\log\log(1/\delta)}{\log(1/\delta)}.$$
\end{lemma}

\begin{proof}
Write $p_Z(z) = \pr(Z=z)$ for the probability
mass function of $Z$.
Let $0<\delta<1/2$ be a parameter to be chosen later. Since
$$C \geq \Eta(Z) = \sum_{z\in G} p_Z(z) \log{1\over p_Z(z)}$$
and
$$\sum_{z\in G} p_Z(z) \log{1\over p_Z(z)} \ge \biggl(\log{1\over \delta}\biggr)\sum_{z\notin A} p_Z(z) = \biggl(\log{1\over\delta}\biggr)\pr(Z\notin A),$$
we have
\begin{equation}\begin{aligned}\label{eqCoverlog}
\pr(Z\notin A) \le {C\over \log(1/\delta)}.
\end{aligned}\end{equation}
Therefore, we can choose $\delta>0$ small enough,
depending only on $C$,
so that $\pr(Z\in A)\neq 0$ and the distribution 
of $X'$ is well defined. Also,
by elementary properties of the map
$h_2(p)= - p\log p - (1-p)\log(1-p)$, 
$p\in(0,1)$, we have that
$$\Eta(\one_{\{Z\in A\}}) \le C{\log\log(1/\delta)\over \log(1/\delta)}.$$
Consequently,
\begin{equation}\begin{aligned}\label{eqloglogoverlog}
\Eta(X\given \one_{\{Z\in A\}})
&= \Eta(X\given Z\in A)\pr(Z\in A) + 
\Eta(X\given Z\notin A)\pr(Z\notin A) \cr
&\ge \Eta (X) - C{\log\log(1/\delta)\over \log(1/\delta)}.
\end{aligned}\end{equation}

Now, let $(X_1, Z_1)$ and $(X_2, Z_2)$ be independent copies of $(X,Z)$, so that
\begin{equation}\begin{aligned}\label{eqcheckpoint}
\Eta(X_1 + X_2) &\ge \Eta(X_1 + X_2\given \one_{\{Z_1\in A\}}, 
\one_{\{Z_2\in A\}}) \cr
&= \Eta(X_1 + X_2 \given Z_1, Z_2\in A) \pr(Z\in A)^2 \cr
&\qquad + \Eta(X_1 + X_2 \given Z_1\in A, Z_2\notin A) 
\pr(Z\notin A)\pr(Z\in A) \cr
&\qquad + \Eta(X_1 + X_2 \given Z_1\notin A, Z_2\in A)
\pr(Z\notin A)\pr(Z\in A) \cr
&\qquad + \Eta(X_1 + X_2 \given Z_1, Z_2\notin A) \pr(Z\notin A)^2.
\end{aligned}\end{equation}
Since $X_1$ and $X_2$ are conditionally independent 
given $(Z_1, Z_2)$, we have
\begin{align*}  
\Eta(X_1 + X_2 \given Z_1\in A, Z_2\notin A) &\ge \Eta(X\given Z\in A),\\
\Eta(X_1 + X_2 \given Z_1\notin A, Z_2\in A) &\ge \Eta(X\given Z\notin A),\\
\Eta(X_1 + X_2 \given Z_1\notin A, Z_2\in A) &\ge \Eta(X\given Z\in A),\\
\Eta(X_1 + X_2 \given Z_1, Z_2\notin A) &\ge \Eta(X\given Z\notin A).
\end{align*}
Thus,
\begin{equation}\begin{aligned}\label{eqadditioncounterpart}
\Eta(X_1 + X_2) &\ge \Eta(X_1 + X_2 \given Z_1, Z_2\in A) \pr(Z\in A)^2 \cr
&\qquad+\Eta(X\given Z\in A)   \pr(Z\in A)\pr(Z\notin A)  \\
&\qquad+\Eta(X\given Z\notin A)\pr(Z\in A)\pr(Z\notin A) \cr 
&\qquad+\Eta(X\given Z\notin A)\pr(Z\notin A)^2\cr
&= \Eta(X\given \one_{\{Z\in A\}}) + \bigl(\Eta(X_1'+X_2') 
- \Eta(X')\bigr)\pr(Z\in A)^2, 
\end{aligned}\end{equation}
and substituting~\eqref{eqCoverlog} and~\eqref{eqloglogoverlog} 
into~\eqref{eqadditioncounterpart} gives~\eqref{eq:firstclaim},
settling the first claim of the lemma.

To obtain the second claim, 
assume without loss of generality that 
$\pr(X\neq 0)$ is nonzero, and
observe that the analysis we performed up to and including~\eqref{eqcheckpoint} is valid with addition replaced by multiplication in the field. 
From there, we note that
\begin{align*}
\Eta(X_1 X_2 \given Z_1\in A, Z_2\notin A) 
&\ge \pr(X\ne 0|Z\notin A)
\Eta(X\given Z\in A),\\
\Eta(X_1 X_2 \given Z_1\notin A, Z_2\in A) 
&\ge \pr(X\ne 0|Z\in A)
\Eta(X\given Z\notin A),\\
\Eta(X_1 X_2 \given Z_1, Z_2\notin A) 
&\ge \pr(X\ne 0|Z\notin A)
\Eta(X\given Z\notin A).
\end{align*}
Thus the counterpart to~\eqref{eqadditioncounterpart} is
$$\frac{\Eta(X_1X_2)}{\pr(X\ne 0)}\ge 
\Eta(X\given \one_{\{Z\in A\}}) + 
\bigl(\Eta(X_1'X_2') - \Eta(X')\bigr)\frac{\pr(Z\in A)^2}{\pr(X\ne 0)},$$
and the second claim follows by 
bounding $\Eta(X\given \one_{\{Z\in A\}})$ as before. 
\end{proof}

We will later use the following statement about symmetric coset progressions.

\begin{theorem}[{\rm\cite[Corollary 1.18]{taovu2008}}] 
\label{thmtorsionjohn}
Let $P$ be a symmetric coset progression of rank $d\ge 0$, and let
$t\ge 1$. There exists a $t$-proper symmetric coset progression $Q$ of rank at most $d$ such
that $P\subseteq Q$ and
$$|Q| \le t^d O(d)^{3d^2/2} |P|.$$
\end{theorem}

Since the coset progressions we will be considering
are not necessarily symmetric, in order to 
apply \Cref{thmtorsionjohn} we will need the following lemma.

\begin{lemma}\label{lemsymmetric}
Let $G$ be an abelian group and let
$H+P\subseteq G$ be a coset progression of rank $d$. There is a symmetric coset progression $H+P'$
of rank at most $d+1$ with $H+P\subseteq H+P'$ and
$$|H+P'| \le O\bigl(|H+P|\bigr).$$
\end{lemma}

\begin{proof}
Let $r_1,\ldots, r_d\in G$ and $N_1, \ldots, N_d$ be positive integers such that
$$ P = \bigl\{a + n_1 r_1 + \cdots + n_d r_d : n_1 \in [0, N_1) ,\ldots, n_d\in [0,N_d)\bigr\}.$$
By possibly increasing each $N_i$ by $1$ (and thereby increasing the size of $|P|$ by no more than
a constant factor) we can assume that all of the $N_i$ are odd.
Then, letting
$$x = a + {N_1 - 1\over 2} r_1 + \cdots + {N_d - 1\over 2} r_d,$$
we see that $P$ is equal to $x+P'$, where $P'$ is symmetric. Hence $H+P$ is contained
in the symmetric coset progression
$$H + P' + \{-x, 0, x\},$$
which is of rank at most $d+1$ and size at most $3|H+P|$.
\end{proof}

\Cref{thmtorsionjohn} supplies us with a $t$-proper generalised arithmetic progressions for any $t$. As
long as $t\ge 3$, this induces a Freiman isomorphism to a discrete box.

\begin{lemma}\label{lemproper}
Let $G$ be any abelian group, $d\ge 0$ be an integer, and suppose that
for some $r_1,\ldots,r_d\in G$ and $N_1,\ldots,N_d\ge 1$,
the symmetric generalised arithmetic progression
$$ Q = \bigl\{n_1 r_1 + \cdots + n_d r_d : n_1 \in [-N_1,N_1] ,\ldots, n_d\in [-N_d, N_d]\bigr\}$$
is $t$-proper for some $t\ge 3$. Then the map $\phi:G\to \ZZ^d$ given by
$$n_1 r_1 + \cdots + n_d r_d \mapsto (n_1,\ldots,n_d)$$
induces a Freiman isomorphism from $Q$ to the discrete box
$B = [-N_1,N_1] \times \cdots\times [-N_d, N_d]$.
\end{lemma}

\begin{proof}
Since $t\ge 1$, $Q$ is proper and hence $\phi$ induces a bijection from $Q$ to $B$. It is clear that for all $(m_1,\ldots,m_d), (m_1',\ldots,m_d'), (n_1,\ldots,n_d), (n_1',\ldots,n_d) \in B^4$, if
$$(m_1,\ldots,m_d) + (n_1,\ldots, n_d)
= (m_1',\ldots,m_d') + (n_1',\ldots,n_d'),$$
then
$$(m_1+n_1)r_1 + \cdots + (m_d+n_d)r_d = (m_1'+n_1')r_1 + \cdots + (m_d'+n_d')r_d.$$
The converse follows from our assumption that $Q$ is $3$-proper applied to the fact that $(m_1+n_1, \ldots, m_d + n_d)$
and $(m_1'+n_1', \ldots, m_d+n_d)$ are both in $[-2N_1, 2N_1]\times\cdots\times[-2N_d, 2N_d]$.
\end{proof}

The next lemma shows that Freiman isomorphisms preserve entropies of sums of random variables.

\begin{lemma}\label{lemfreiman}
Let $G$ and $H$ be abelian groups. Let $A\subseteq G$ and
$B\subseteq H$ be Freiman isomorphic; that is, there exists a Freiman isomorphism $\phi:A\to B$.
Let $X$ be a $G$-valued random variable whose support is contained in $A$
and let $X' = \phi(X)$.
Then,
$$\Eta(X_1 + X_2) = \Eta(X_1' + X_2'),$$
where $X_1,X_2$ are independent copies of $X$, and $X_1',X_2'$ are 
independent copies
of $X'$.
\end{lemma}

\begin{proof}
For any $x_1,x_2\in A$,
$$\pr\bigl( X_1' + X_2' = \phi(x_1) + \phi(x_2) \bigr) 
= \pr(X_1 + X_2 = x_1 + x_2),$$
since $\phi$ is a Freiman isomorphism. The result follows.
\end{proof}

The following proposition from~\cite{tao2010} states that
any discrete random variable with small doubling is close
to a uniform random variable on a coset progression.

\begin{proposition}[{\rm\cite[Proposition~5.2]{tao2010}}]
\label{proptaofivetwo}
Let $G$ be an abelian group and suppose $X$ is a
$G$-valued discrete random variable such that
$$\Eta(X+X') \le \Eta(X) + K,$$
where $X'$ is an independent copy of $X$ and $K$ is some constant. Then
we may express $X = U+Z$, where 
$Z$ has $\Eta(Z) = O_K(1)$ and
$U$ is 
uniformly distributed on a coset progression 
$H+P\subseteq G$ of rank $O_K(1)$ and
cardinality $O\bigl(2^{\Eta(X)}\bigr)$.
\end{proposition}

We are now ready to prove \Cref{thmFpEPI}.

\begin{proof}[Proof of \Cref{thmFpEPI}.]
Let $0<\eps<1/2$. 
We will establish the existence of $k_\eps, K_\eps\in \RR$ such that,
for any prime $p$ with $\log p > k_\eps + K_\eps$ and any 
$\FF_p$-valued random variable $X$ satisfying
\begin{equation}
k_\eps < \Eta(X) < \log p - K_\eps,
\label{eq:toprove}
\end{equation}
one has
$$\Eta(X_1+X_2) \ge \Eta(X) + \frac12 - \eps,$$
where $X_1,X_2$ are independent copies of $X$.

We shall defer the choices of $k_\eps,$ $K_\eps$ and
an additional parameter $0<\delta<1/2$ until later. 
Assume 
that the claim is false, so that~\eqref{eq:toprove} holds but
$$\Eta(X_1+X_2) < \Eta(X) + {1\over 2} - \eps.$$
Then by \Cref{proptaofivetwo}, we can write $X = U+Z$, where
$\Eta(Z) = O(1)$ and
$U$ is uniformly distributed on a coset progression $H+P$ 
of rank $O(1)$ and
cardinality $O(2^{\Eta(X)})$.
We have
$$\Eta(X) \ge \Eta(U+Z\given Z) = 
\Eta(U\given Z) \ge \log|H+P| - \Eta(Z).$$

Picking a $C\geq \Eta(Z)$,
defining 
$A = \bigl\{z\in \FF_p : \pr(Z=z)\ge \delta\bigr\},$
and letting $X'$ have the distribution
of $X$ conditioned on $\{Z\in A\}$,
\Cref{lemZsupport} gives
$$
\Eta(X_1 + X_2) - \Eta(X)
\ge \bigl(\Eta(X_1' + X_2') - \Eta(X')\bigr)\biggl(1-\frac{C}
{\log(1/\delta)}\biggr)^2 - C\frac{\log\log(1/\delta)}{\log(1/\delta)}.
$$
Choosing $\delta<\delta_0$ small enough in terms of 
$\eps$, we have
$$\frac12-\eps > \Eta(X_1 + X_2) - \Eta(X)
\ge \bigl(\Eta(X_1' + X_2') - \Eta(X')\bigr)\biggl(1-\frac{\eps/2}{1-\eps}\biggr)-\frac\eps4.$$
Hence to obtain a contradiction we need to show that
\begin{equation}
\label{WTSXprime}\Eta(X_1' + X_2')- \Eta(X') \ge {1/2-3\eps/4\over 1-(\eps/2)/(1-\eps)}=
{1\over 2} - {\eps\over 2},
\end{equation}
where $X_1'$ and $X_2'$ are independent copies of $X'$.

Note that $X'$ takes values in the set $A+H+P$ and,
since $|A| \le 1/\delta$, we can embed
$A+H+P$ in a coset progression $H+P'$ of rank at most $O_\delta(1)$ and size
$O_\delta\bigl(|H+P|\bigr)$. By invoking \Cref{lemsymmetric} and 
then \Cref{thmtorsionjohn}
we can find a symmetric coset progression $H+Q$ of rank $O_\delta(1)$
such that $H+P'\subseteq H+Q$ and $|H+Q| \le K_\delta |H+P|$ for some $K_\delta$ depending on $\delta$
(which depends only on $\eps$).
We can assume that $H+Q$ is $3$-proper by increasing $K_\delta$ if needed.

Let $K_\eps$ be a large absolute constant such that $\log K_{\eps} \geq 
C+ \log K_\delta$.
There are {\it a priori} two possibilities for the subgroup 
$H\subseteq \FF_p$; either $H=\{0\}$ or $H=\FF_p$. However,
by our hypothesis on $\Eta(X)$,
\begin{equation}\begin{aligned}
\log |H+Q| &\le \log |H+P| + \log K_\delta \cr
&\le \Eta(X) + \Eta(Z) + \log K_\delta \cr
&\leq \Eta(X) + \log K_\eps < \log p,
\end{aligned}\end{equation}
so $H+Q\ne \FF_p$ and thus $H\ne \FF_p$. Hence we have $H=\{0\}$ and $X'$ is supported on the
generalised arithmetic progression $Q$.

Let $d = O_\delta(1)$ be the rank of $Q$. Since $Q$ is $3$-proper, by \Cref{lemproper} it is
Freiman isomorphic to a box $[-N_1,N_1]\times\cdots\times [-N_d,N_d]\subseteq
\ZZ^d$. Let $\phi$ be this Freiman isomorphism, and let $X'' = \phi(X')$. 
By \Cref{lemfreiman}
we see that proving~\eqref{WTSXprime} is equivalent to showing that
$$\Eta(X_1'' + X_2'') \ge \Eta(X'') + {1\over 2} - {\eps\over 2},$$
where $X_1''$ and $X_2''$ are independent copies of $X''$. 
But $\ZZ^d$ is torsion-free, so appealing 
to~\cite[Theorem~1.9]{tao2010},
this is true so long as $\Eta(X'')$ is sufficiently large in terms of $\eps$, which in turn follows from our assumption that $\Eta(X) > k_\eps$ is sufficiently large.
\end{proof}

\paragraph{Independent but non-identically distributed random variables}
\Cref{thmFpEPI} was only stated and proved for 
the sum of independent random variables $X,X'$ with the same distribution.
If $X,Y$ are independent but not identically distributed, 
\Cref{thmFpEPI} implies a bound for $\Eta(X+Y)$ but with a 
worse constant. Its proof uses 
the \textit{entropic Ruzsa distance} between two random variables 
$X$ and $Y$ taking values in the same abelian group,
defined as
$$\dR(X,Y) = \Eta(X'-Y') - \frac12\Eta(X) - \frac12\Eta(Y),$$
where $X',Y'$ are independent, with the
same marginal distributions as $X,Y$.
The entropic Ruzsa distance
 satisfies~\cite{tao2010}
the triangle inequality
$$\dR(X,Z) \le \dR(X,Y) + \dR(Y,Z),$$
as well as the sum-difference inequality
$$\dR(X,-Y) \le 3\dR(X,Y).$$

\begin{theorem}\label{thmnoniid}
Let $G = \FF_p$ or $G=\RR$, let $0<\eps<1/2$, and let $X$ and $Y$ be 
independent discrete random variables taking values in $G$. Suppose that
$\Eta(X)$
is large enough depending on $\epsilon$ and,
if $G = \FF_p$ for some prime $p$,
then suppose further that $\log p\ge k_\eps + K_\eps$ and
$$k_\eps < \Eta(X) < \log p - K_\eps,$$
where $k_\eps$ and $K_\eps$ are parameters given by \Cref{thmFpEPI}.
Then
$$\Eta(X + Y) \geq \frac{1}{2}\Eta(X) + \frac{1}{2}\Eta(Y) + \frac{1}{8} - \epsilon.$$
\end{theorem}

\begin{proof} Let $X'$ be an independent copy of $X$.
By applying either \Cref{thmFpEPI} (in the case $G=\FF_p$) or \Cref{thmtaoEPI} (in the case $G=\RR$), we see that
\begin{equation} \label{triangleineqs}
\frac{1}{2} - \epsilon \leq \Eta(X+X') - \Eta(X) = \dR(X,-X) \leq \dR(X,Y) + \dR(X,-Y) \leq 4\dR(X,Y).
\end{equation}
The claim follows after replacing $Y$ with $-Y$.
\end{proof}

Interestingly, the same limiting constant of $1/8$ was obtained, using different methods, for independent and not necessarily identically distributed integer-valued random variables in \cite[Theorem 1]{hat}. Nevertheless, in view of Shannon's continuous entropy power 
inequality~\cite{shannon:48,stam:59}, we believe 
that the constant 1/8 in \Cref{thmnoniid}
can be improved to 1/2, as is the case 
where they are identically distributed. 

\begin{conjecture}
The constant $1/8$ in \Cref{thmnoniid} can be improved to $1/2$.
\end{conjecture}

\section{Min-entropy and sum-product phenomena}

Let $F$ be a field, $P\subseteq F^3$ be a finite set of points in 
three-dimensional space over $F$, and $Q$ be a finite set of planes 
(affine subspaces of codimension $1$) in $F^3$. An \textit{incidence} between $P$ and $Q$ is a pair $(x,h)\in P\times Q$ such that the point $x$ lies on the plane $h$. A central problem in incidence geometry is the study of the number
$$I(P,Q) = \big|\bigl\{ (x,h)\in P\times Q : \textrm{the point}\ x\ \textrm{lies on the plane}\ h\bigr\}\big|$$
of incidences between $P$ and $Q$.

We shall use the following prime-field analogue 
of Rudnev's point-plane incidence bound
due to Koll\'ar~\cite{kollar2015}.
We record a slight strengthening due to de Zeeuw, 
see~\cite{dezeeuw2016} for a short proof.

\begin{lemma}[{\rm\cite[Theorem 4.1]{dezeeuw2016}}]
\label{lempointplane} Let $F$ be a field of characteristic $p$.
Let $P$ be a set of points in $F^3$, and let $Q$ be a set of planes in $F^3$. If
\begin{enumerate}
\item $|P|\le |Q|$;
\item either $p=0$ or $|P|\ll p^2$; and
\item no line in $F^3$ contains $s$ points of $P$ and is contained in $t$ planes of $Q$,
\end{enumerate}
then
$$I(P,Q) \ll |P|^{1/2} |Q| + t|P| + s|Q|.$$
\end{lemma}

Using this, the following lower 
bound on the collision entropy
can be established.

\begin{lemma}\label{lemkohetal}
Let $F$ be a field of characteristic $p$.
Let $U_A$, $U_B$, and $U_C$ be independent uniform random 
variables on finite subsets
$A$, $B$, and $C$ of $F$, respectively. Assume that $|A|=|B|=|C|$
and, if $p$ is positive, suppose further that $|A| \ll p^{2/3}$.
Then
$$\Eta_2\bigl( U_A(U_B+U_C)\bigr) \ge \frac32\log|A| - O(1).$$
\end{lemma}

\begin{proof}
We extract the relevant argument from the proof 
of~\cite[Theorem~2.1]{kohetal2020}.
Let $P\subseteq F^3$ be the set of points
$$\bigl\{ (a, b', ac) : (a,b',c)\in A\times B \times C\bigr\},$$
and let $Q$ be the collection of planes
$$\bigl\{ bx - a'y + z = a'c' : (a', b, c')\in A\times B\times C\bigr\}.$$
Note that if an element of $P$ lies in a plane from the set $Q$, then
$$ba - a'b' + ac = a'c',$$
hence
$$a(b+c) = a'(b'+c').$$
Therefore, letting
$$N = \Bigl|\bigl\{ (a,b,c,a',b',c')\in A\times A\times B\times B\times C\times C : a(b+c) = a'(b'+c')\bigr\}\Bigr|,$$
we have $N = I(P,Q)$. The first condition of \Cref{lempointplane} is 
satisfied trivially, and the second is satisfied by our assumption 
(in the case that $p>0$) that $\max\bigl\{ |A|, |B|, |C|\bigr\} \ll p^{2/3}$. 
For the third condition we claim that, with the choices 
$s = \max\bigl\{|A|, |B|\bigr\}+1$ and $t=1$, it is satisfied by $P$ and $Q$.
This claim from the observation that, if we project any nonvertical line 
in $F^3$ to the plane $z=0$, the result is a line that can intersect at 
most $s-1$ points of $P$, hence the original line cannot intersect 
$s$ points of $P$. On the other hand, a plane in $Q$ has normal vector 
equal to $(b, a', 1)$ for some $a'\in A$ and $b\in B$, which, in particular,
means it cannot contain any vertical line. Hence the conclusion 
of \Cref{lempointplane} holds, and we have
$$N \ll \bigl( |A|\cdot |B|\cdot |C|\bigr)^{3/2} + \Bigl(\max\bigl\{|A|, |B|\bigr\} +1\Bigr) |A|\cdot |B|\cdot |C| \ll |A|^{9/2}.$$
This count implies the claimed bound on the collision entropy, since
\begin{equation}\begin{aligned}
\Eta_2\bigl( U_A(U_B+U_C)\bigr)
&= -\log \pr_{a,a'\in A, b,b'\in B, c,c'\in C}
\bigl( a(b+c) = a'(b'+c') \bigr)\cr
&=  6 \log|A| - \log N\cr
&\ge \frac32 \log |A| - O(1), 
\end{aligned}\end{equation}
where the probability is with respect to uniformly and independently
random 
$a,a'\in A$, $b,b'\in B$ and $c,c'\in C$.
\end{proof}

As a corollary, we obtain a lower bound concerning the expander $A(B+C)$.

\begin{corollary} 
\label{cor:ABC}
Let $F$ be a field of characteristic $p$.
Let $A$, $B$, and $C$ be finite subsets of $F$ with $|A|=|B|=|C|$
and, if $p>0$, assume that $|A|\ll p^{2/3}$. Then
$\bigl| A(B+C)\bigr| \gg |A|^{3/2}$.
\end{corollary}

\begin{proof}
Let $U_A$, $U_B$, and $U_C$ be independent uniform random variables on $A$, $B$, and $C$, respectively. Then,
$$\log\bigl| A(B+C)\bigr| \ge \Eta\bigl( U_A(U_B+U_C) \bigr)
\ge \Eta_2\bigl( U_A(U_B+U_C)\bigr) \ge \frac32 \log |A| - O(1),$$
and the claim follows.
\end{proof}

The combinatorial bound of \Cref{cor:ABC} actually also 
follows from the more general bounds 
in~\cite[Corollary~20]{yazicietal2017},
established using combinatorial methods.
Over finite fields, 2/3 appears to be the best known exponent.
Over $\RR$, the best exponent (up to logarithmic factors)
is $32/21\approx 1.523$, 
due to Bloom~\cite{bloom2025}. 
It is in fact believed that $\bigl|A(A+A)\bigr| \gg |A|^{2-o(1)}$ and,
indeed, this would follow from the conjectured sum-product 
bound $\max\bigl\{ |A+A|, |A\cdot A|\bigr\} \ge |A|^{2-o(1)}$.

Next, we generalize \Cref{lemkohetal} to non-uniform random variables by expressing our starting random variable as a mixture of uniforms. To our knowledge, the following lemma first appeared in a paper of
Chor and Goldreich~\cite[Lemma 5]{CG88}, but we record a version given by Vadhan which is more convenient for our purposes.

\begin{lemma}
[{\rm\cite[Lemma 6.10]{vadhan2012}}\/]
\label{lemconvexcomb}
Let $X$ be a random variable with finite support $A$,
such that and $\Etamin(X) \ge m$, where $2^m$ is an integer. 
Then there is an integer $s$, nonnegative real numbers $p_1,\ldots,p_s$, 
and random variables $U_1,\ldots,U_s$ such that:
\begin{enumerate}
\item $\sum_{i=1}^s p_i = 1$;
\item each $U_i$ is uniformly distributed on some set $A_i\subseteq A$ with $|A_i| = 2^m$; and
\item $\pr(X=a) = \sum_{i=1}^s p_i \pr(U_i = a)$ for all $a\in A$.
\end{enumerate}
\end{lemma}

Using this, we can give a lower bound on the entropy of any 
random variable of the form $X(Y+Z)$. 

\begin{lemma}
\label{lemHminlowerbound}
Let $F$ be a field, and let $X$, $Y$, and $Z$ be independent finitely supported random variables taking values in $F$. Suppose that
$$\min\bigl\{ \Etamin(X), \Etamin(Y), \Etamin(Z)\bigr\} \ge m,$$
where $2^m$ is an integer and,
if the characteristic of the field is $p>0$, assume further 
that $m\le (2/3)\log p - O(1)$. Then
$$
\Eta\bigl(X(Y+Z)\bigr) \ge \frac32 m - O(1).
$$
\end{lemma}

\begin{proof}
Let $s$, $0\le p_1,\ldots,p_s\le 1$, and $U_1,\ldots,U_s$, be as
in the conclusion of \Cref{lemconvexcomb} applied to $X$.
Let $V$ be the random variable on $[s]$ with $\pr(V=i) = p_i$ 
for all $i\in [s]$, independent of $U_1,\ldots,U_s$, so that
$X$ can be expressed as $X=U_V$.
Repeating the same procedure with $Y$ and $Z$, gives
corresponding 
$s'$, $0\le p'_1,\ldots,p'_{s'}\le 1$, $U'_1,\ldots,U'_{s'}$, $V'$
and
$s''$, $0\le p''_1,\ldots,p''_{s''}\le 1$, $U''_1,\ldots,U''_{s''}$, $V''$.
Then we can bound
\begin{equation}\begin{aligned}
\Eta\bigl(X(Y+Z)\bigr) &\ge \Eta\bigl(X(Y+Z) \given V, V', V''\bigr)\cr
&= \ex_{(i,j,k)\sim (V,V',V'')} \big[\Eta\bigl(U_i(U_j'+U_k'')\bigr)\big]\cr
&\ge \ex_{(i,j,k)\sim (V,V',V'')} \biggl( \frac32\log |A_i| - O(1)\biggr) \cr
&= \frac32 m - O(1), 
\end{aligned}\end{equation}
where the inequality follows by \Cref{lemkohetal}.
\end{proof}

The following lemma gives 
an upper bound for the entropy $\Eta\bigl( X(Y+Z)\bigr)$. 
The result is a slightly modified version of a bound
due to M\'ath\'e and O'Regan, which was included
as Proposition~4.1 in the original preprint 
version of~\cite{matheoregan}. As it does not appear
in the published version of their paper,
we offer a proof for the sake of completeness.

\begin{lemma}\label{lemMO}
For any random variables $X$, $Y$, and $Z$ taking values in a field,
\begin{align}\label{eqMOone}
\Eta\bigl( X(Y+Z)\bigr) \le \Eta(XY, XZ) + &\Eta(X, Y+Z) - \Eta(X,Y,Z) \cr
&+ \pr(X=0)\Eta(Y,Z\given Y+Z,X=0).
\end{align}
In particular, when $X$, $Y$, and $Z$ are independent 
and identically distributed, 
\begin{align}\label{eqMOtwo}
\Eta\bigl( X(Y+Z)\bigr) &\le 2\Eta(XY) + \Eta(X+Y) - 2\Eta(X) + \pr(X=0)\bigl(2\Eta(X) - \Eta(X+Y)\bigr) \cr
&= 2\Eta(XY) - \pr(X\ne 0)\bigl(2\Eta(X) - \Eta(X+Y)\bigr)\cr
\end{align}
\end{lemma}

\begin{proof}
It is clear how~\eqref{eqMOone} implies~\eqref{eqMOtwo}; 
we shall prove the former.
The pairs $(X,Y+Z)$ and $(XY, XZ)$ separately determine $X(Y+Z)$. Hence by the submodularity inequality for entropy~\cite{tao2010},
\[\Eta\bigl(X(Y+Z)\bigr)
\le \Eta(X, Y+Z) + \Eta(XY, XZ) - \Eta(X,Y+Z,XY,XZ).\]
But $(X,Y+Z,XY,XZ)$ determines $(X,Y,Z)$ when $X\ne 0$,
so since 
\begin{align*}
\Eta(X,Y,Z \given X,Y+Z,XY,XZ) &= \pr(X=0) \Eta(Y,Z\given Y+Z, X=0)  \\
& \quad - \pr(X=1) \Eta(Y,Z\given Y+Z,XY,XZ, X=1),
\end{align*}
we see that
\begin{align}
\Eta(X,Y+Z,XY,XZ)
&= \Eta(X,Y,Z) - \Eta(X,Y,Z \given X,Y+Z,XY,XZ) \cr
&= \Eta(X,Y,Z) - \pr(X=0) \Eta(Y,Z\given Y+Z, X=0) \cr
\end{align}
from which~\eqref{eqMOone} follows.
\end{proof}

We can now prove \Cref{thmminentropy}.

\begin{proof}[Proof of \Cref{thmminentropy}.]
If $2\Eta(X)\ge 3\Etamin(X)$, then
$$\Eta(X+X')\ge \Eta(X) \ge \frac23\Eta(X)
+ \frac{\Etamin(X)}{2}\ge \frac{4 \Eta(X) + 3\Etamin(X)}{6}.$$
So for the remainder of the proof, we assume that $2\Eta(X) < 3\Etamin(X)$. 

Let $X$ be finitely supported in a field $F$. If $F$ has characteristic $p>0$, then assume that $\Etamin(X) \le (2/3)\log p - O(1)$. 
Let $X'$ be an independent copy of $X$, 
and write $m = \lfloor \Etamin(X)\rfloor$. 
Applying Lemma~\ref{lemHminlowerbound} and
Lemma~\ref{lemMO} gives
\begin{align}
\max\bigl\{ \Eta(X+X'), \Eta(XX') \bigr\}
&\ge
\frac{4\pr(X\ne 0)\Eta(X)+3m}
     {4+2\pr(X\ne 0)}
-O(1) \cr
&\ge
\frac{4\pr(X\ne 0)\Eta(X)+3\Etamin(X)}{6}
-O(1).
\end{align}
To complete the proof, we simply need to show that
$\pr(X=0)\Eta(X) = O(1)$. Since we assumed that $\Etamin(X) > (2/3)\Eta(X)$, we have
\[
\pr(X=0)\Eta(X)
\le \Eta(X)2^{-\Etamin(X)}
\le \Eta(X)2^{-2\Eta(X)/3}.
\]
This function of $\Eta(X)$ attains a global maximum, so the proof is complete.
\end{proof}

It is clear from the proof above that if the 
coefficient $3/2$ is \Cref{lemkohetal} were improved to 2
(perhaps with nonconstant error terms), then we would obtain an exponent 
of $4/3$ for the combinatorial sum-product problem; over $\FF_p$ this 
would be a new result. Furthermore, this would also imply the
combinatorial bound $\bigl|A(A+A)\bigr| \gg |A|^{2-o(1)}$
mentioned earlier. However, this conjectured bound appears to be
out of the range of current techniques.

\paragraph{Randomness extraction}
We conclude this section with a few remarks on \Cref{lemHminlowerbound} and its link to randomness extraction.
Suppose there is a source of randomness with min-entropy $m$ from 
which one may repeatedly obtain independent samples, from which they
are asked to produce a random variable that is close in total 
variation distance to a uniform random variable on $\{0,1\}^n$.
A theorem of Barak, Impagliazzo, and Wigderson~\cite{biw2006} 
describes the performance of one such scheme.

\begin{theorem}[{\rm\cite[Theorem 1.1]{biw2006}}\/]\label{thmBIW}
There exists an absolute constant $C$ such that for any $\delta > 0$, given $(1/\delta)^{C}$ independent samples of a random variable with min-entropy $m \ge\delta n$, one may produce a random variable that is at total variation distance at most $1/2^{\Omega(n)}$ from a uniform random variable on $\{0,1\}^n$.
\end{theorem}

Their idea is to embed the problem into the prime field $\FF_p$ for some 
prime $p\in [2^n, 2^{n+1})$, and then iterate the function 
$(X,Y,Z)\mapsto XY+Z$. A key intermediate step,
\cite[Lemma~3.1]{biw2006} in the proof, states that for some 
small $\eps>0$ the random variable $XY+Z$ is at total variation distance 
less than $1/2^{\eps m}$ from a random variable with min-entropy $(1+\eps)m$. 
Our \Cref{lemHminlowerbound} improves upon this, since $X(Y+Z)$ requires 
the same number of operations to compute as $XY+Z$, there is no statistical 
error when measuring the min-entropy of the output, and we are able to 
take $\eps = 1/2$ whereas the corresponding constant obtained 
in~\cite{biw2006} was small and unspecified.

By appropriately modifying the Barak--Impagliazzo--Wigderson argument 
to iterate the function $(X,Y,Z) \mapsto X(Y+Z)$ instead of 
$(X,Y,Z)\mapsto XY+Z$, the quantitative improvement offered by
\Cref{lemHminlowerbound} implies the following corollary.

\begin{corollary}
\Cref{thmBIW} holds with $C = \log_{3/2}(3)\approx 2.71$.
\end{corollary}

\section{Improved coefficients over the real numbers}

The aim of this brief section is to prove \Cref{thmminentropyreal}. 
Specifically, we will show that the expression $\bigl(4\Eta(X) + 3m\bigr) / 6$ 
in \Cref{thmminentropy} can be improved to $\bigl(4\Eta(X) + 2m\bigr)/5$ 
(at the cost of an extra logarithmic term) for the special case $F=\RR$. 

Whereas in the previous section we bounded the entropy of the random 
variable $X(Y+Z)$ from above and below, here we work with 
a random variable of the form $(X+Y)(Z+W)$. For the
lower bound, we will use the following inequality due to
Roche-Newton and Rudnev.

\begin{proposition}[{\rm\cite[Proposition~4]{roru2015}}]
\label{proproru}
Let $P\subseteq \RR^2$ be a finite set of points in the plane. Selecting $(a,c)$, $(b,d)$, $(a',c')$, and $(b',d')$ uniformly at random from $P$, we have
$$
\pr\bigl( (a-b)(c-d) = (a'-b')(c'-d')\ne 0\bigr) 
\ll \frac{\log|P|}{|P|}.
$$
\end{proposition}

Using this proposition, we prove the following 
analogue of \Cref{lemHminlowerbound}.

\begin{lemma}\label{lemHminlowerboundreal}
Let $X$, $Y$, $Z$, and $W$ be independent finitely supported random variables taking values in $\RR$. Suppose that
$$\min\bigl\{ \Etamin(X), \Etamin(Y), \Etamin(Z), \Etamin(W)\bigr\} \ge m,$$
where $2^m$ is an integer. Then
$$\Eta\bigl( (X+Y)(Z+W)\bigr) \ge 2m - O(\log m).$$
\end{lemma}

\begin{proof}
First we consider random variables that are uniform on finite sets 
of the same size, as in \Cref{lemkohetal}.
Let $A,B,C,D$ be finite subsets of $\RR$ with
$|A| = |B| = |C| = |D|$ and write
$$S = A\cup B\cup C\cup D \cup -A \cup -B \cup -C\cup -D.$$
Letting $P = S^2$, we have $|P| \le 64|A|^2$. This yields the bound
\begin{equation}\begin{aligned}
\Bigl| \bigl\{ a,a'\in A, b,&b'\in B, c,c'\in C, d,d'\in D
: (a+b)(c+d) = (a'+b')(c'+d')\bigr\}\Bigr|\cr
&\le \Bigl| \bigl\{ a,a',b,b',c,c',d,d'\in S
: (a-b)(c-d) = (a'-b')(c'-d')\bigr\}\Bigr| \cr
&\le |P|^3 \log |P|\cr
&\ll |A|^6 \log |A|,
\end{aligned}\end{equation}
which implies that
$$\Eta\bigl( (U_A+U_B)(U_C+U_D)\bigr)\ge \Eta_2\bigl( (U_A+U_B)(U_C+U_D) \bigr)
\ge 2\log |A| - O\bigl(\log|A|\bigr).$$
The full claim is then proved by using \Cref{lemconvexcomb} 
to express $X$, $Y$, $Z$, and $W$ as mixtures of uniform random variables,
and proceeding exactly as in the proof of \Cref{lemHminlowerbound}.
\end{proof}

Next we establish a corresponding
upper bound for $\Eta\bigl((X+Y)(Z+W)\bigr)$.

\begin{lemma}\label{lempohoata}
Let $X,Y,Z,W$ be discrete random variables taking values in a field. Let $E$ be the event that $XZ\ne 0$ and let $E'$ be the event that $X+Y = 0$ and $Z+W=0$. Then
\[
\begin{aligned}\label{eq:pohoatageneral}
\Eta\bigl((X+Y)(Z+W)\bigr)
\le {}&
\Eta(X+Y,Z+W)+\Eta(XZ,XW,YZ) \notag\\
&-\pr(E)\Eta(X,Y,Z,W\given E) \notag +\pr(E\cap E') \notag \Eta(X,Z \given XZ, E\cap E').
\end{aligned}
\]
If $X,Y,Z,W$ are independent and identically
distributed, then
\[
\begin{aligned}\label{eq:pohoataiid}
\Eta\bigl((X+Y)(Z+W)\bigr)
\le 
2&\Eta(X+Y)+3\Eta(XY) \cr
&\qquad-2q(1+q)\Eta(X) +2q h_2(q)+1,\cr
\end{aligned}
\]
where $q=\pr(X\ne 0)$.
In particular, if $\pr(X=0)=0$, then
\[
\Eta\bigl((X+Y)(Z+W)\bigr)
\le
2\Eta(X+Y)+3\Eta(XY)-4\Eta(X)+1.
\]
\end{lemma}
\begin{proof}
Let $V=(X+Y)(Z+W)$.
Conditional on $E=\{XZ\ne0\}$, both $(X+Y,Z+W)$ and $(XZ,XW,YZ)$ determine $V$, since
\[
YW=\frac{(XW)(YZ)}{XZ}.
\]
Thus, by the submodularity inequality~\cite{tao2010}
\begin{align} \nonumber
\Eta(V\given E)
+\Eta(X+Y,Z+W,XZ,XW,YZ\given E)
\le{}&
\Eta(X+Y,Z+W\given E)\\ \label{HVgivEbound}
&+\Eta(XZ,XW,YZ\given E).
\end{align}

Moreover, on $E$, the tuple
\[
(X+Y,Z+W,XZ,XW,YZ)
\]
determines $(X,Y,Z,W)$ unless $X+Y=Z+W=0$. Indeed, if $X+Y\ne0$, then
\[
Z=\frac{XZ+YZ}{X+Y},
\]
and if $Z+W\ne0$, then
\[
X=\frac{XZ+XW}{Z+W}.
\]
When $X+Y=Z+W=0$, we have
\[
(XZ,XW,YZ)=(XZ,-XZ,-XZ).
\]
Consequently,
\begin{align} \nonumber
&\Eta(X+Y,Z+W,XZ,XW,YZ\given E)\\ \nonumber &\qquad= \Eta(X,Y,Z,W\given E) -\Eta(X,Y,Z,W\given X+Y,Z+W,XZ,XW,YZ,E)\\ \nonumber &\qquad= \Eta(X,Y,Z,W\given E)\\ \label{Hofalot}  &\qquad\quad- \pr(X+Y=0,Z+W=0\given E) \Eta(X,Z\given XZ,E,X+Y=0,Z+W=0). 
\end{align}

Since $V$ is determined by $(X+Y,Z+W)$, we have
\[
\Eta(V\given E^c)
\le
\Eta(X+Y,Z+W\given E^c).
\]
Therefore,
\begin{align*}
\Eta(V)
&\le \Eta(V,\one_E)\\
&=\Eta(\one_E)
  +\pr(E)\Eta(V\given E)
  +\pr(E^c)\Eta(V\given E^c)\\
&\le \Eta(\one_E)
  +\pr(E)\Eta(V\given E)
  +\pr(E^c)\Eta(X+Y,Z+W\given E^c).
\end{align*}
Applying \eqref{HVgivEbound} to $\Eta(V\given E)$ and then \eqref{Hofalot}, we obtain
\begin{align*}
\Eta(V)\le{}&
\Eta(\one_E)
+\pr(E)\Eta(X+Y,Z+W\given E)\\
&+\pr(E^c)\Eta(X+Y,Z+W\given E^c)
+\pr(E)\Eta(XZ,XW,YZ\given E)\\
&-\pr(E)\Eta(X,Y,Z,W\given E)\\
&+\pr(E,X+Y=0,Z+W=0)\\
&\qquad\qquad\cdot
\Eta(X,Z\given XZ,E,X+Y=0,Z+W=0).
\end{align*}
Now
\begin{align*}
&\pr(E)\Eta(X+Y,Z+W\given E)
+\pr(E^c)\Eta(X+Y,Z+W\given E^c)\\
&\qquad=
\Eta(X+Y,Z+W\given\one_E)
\le
\Eta(X+Y,Z+W).
\end{align*}
Also, since $\one_E$ is determined by $XZ$, and hence by
$(XZ,XW,YZ)$,
\begin{align*}
\Eta(XZ,XW,YZ)
&=\Eta(\one_E)+\Eta(XZ,XW,YZ\given\one_E)\\
&=\Eta(\one_E)
+\pr(E)\Eta(XZ,XW,YZ\given E)\\
&\qquad
+\pr(E^c)\Eta(XZ,XW,YZ\given E^c),
\end{align*}
and consequently
\[
\Eta(\one_E)
+\pr(E)\Eta(XZ,XW,YZ\given E)
\le
\Eta(XZ,XW,YZ).
\]
Combining the last two estimates gives
\begin{align*}
\Eta(V)\le{}&
\Eta(X+Y,Z+W)+\Eta(XZ,XW,YZ)\\
&-\pr(E)\Eta(X,Y,Z,W\given E)\\
&+\pr(E,X+Y=0,Z+W=0)\\
&\qquad\qquad\cdot
\Eta(X,Z\given XZ,E,X+Y=0,Z+W=0),
\end{align*}
which proves the general statement.

Now suppose that $X,Y,Z,W$ are independent and identically distributed.
Then
\[
\Eta(X+Y,Z+W)=2\Eta(X+Y)
\qquad\hbox{and}\qquad
\Eta(XZ,XW,YZ)\le3\Eta(XY),
\]
and letting $q = \pr(X\ne 0)$, we also have
\[
\pr(E)\Eta(X,Y,Z,W\given E)
=
q^2
\bigl(2\Eta(X\given X\ne0)+2\Eta(X)\bigr)
=
2q(1+q)\Eta(X)
-2qh_2(q).
\]

Finally, write
$r=\pr(X+Y=0\given X\ne0)$.
Then
\begin{align*}
\pr(E,X+Y=0,&Z+W=0)
\Eta(X,Z\given XZ,E,X+Y=0,Z+W=0)\\
&\le q^2r^2\Eta(X\given E,X+Y=0)\\
&\le
2q^2r^2\log
\sum_{x\ne0}
\sqrt{\frac{\pr(X=x)\pr(X=-x)}{qr}}\\
&\le
q^2r^2\log\frac{q}{r}\\
&=
q^4\left(\frac rq\right)^2
\log\frac q r\\
&\le
\frac{1}{2e\log_e2} \cr
&< 1.
\end{align*}
Substitution now gives
\begin{align*}
\Eta\bigl((X+Y)(Z+W)\bigr)
\le{}&
2\Eta(X+Y)+3\Eta(XY)\\
&-2\pr(X\ne0)\bigl(1+\pr(X\ne0)\bigr)\Eta(X)\\
&+2\pr(X\ne0)h_2\bigl(\pr(X\ne0)\bigr)+1;
\end{align*}
and if $\pr(X=0)=0$, this reduces to
\[
\Eta\bigl((X+Y)(Z+W)\bigr)
\le
2\Eta(X+Y)+3\Eta(XY)-4\Eta(X)+1,
\]
which finishes the proof.
\end{proof}

The inequalities 
in Lemmas~\ref{lemHminlowerboundreal}
and~\ref{lempohoata}
yield an improved version of \Cref{thmminentropy} over the reals.

\begin{proof}[Proof of \Cref{thmminentropyreal}.]
Let $X$ be a finitely supported random variable in $\RR$ and
$X'$ be an independent copy of $X$. 
As in the proof of \Cref{thmminentropy}, letting $m=\lfloor 
\Etamin(X)\rfloor$ and 
recalling our assumption that $\pr(X=0) = 0$, we can 
apply
Lemma~\ref{lemHminlowerboundreal} followed by Lemma~\ref{lempohoata}
to obtain
$$\max\bigl\{ \Eta(X+X'), \Eta(XX')\bigr\}
\ge \frac{4\Eta(X) + 2\Etamin(X)}{5} - O\bigl(\log\Etamin(X)\bigr),$$
as required.
\end{proof}

A modified version of \Cref{lempohoata} appears in the 
blog post~\cite{pohoata2026}, where  it is stated that \Cref{proproru} can be generalised to arbitrary fields. If this is true, then one could remove the restriction $F=\RR$ from \Cref{thmminentropyreal}, thereby rendering \Cref{thmminentropy} obsolete. However, we were unable to adapt all the details of Roche-Newton and Rudnev's proof to the case of positive characteristic.

\section{A weak sum-product theorem for entropy}

While the entropic sum-product conjectures of~\cite{goh2024} remain 
unresolved, \Cref{thmminentropy} can be used to derive a weak sum-product 
theorem purely in terms of Shannon entropy. The following proposition is what 
we shall use to relate the min-entropy with Shannon entropy.

\begin{proposition}\label{propmaxprob}
Let $X$ be a random variable taking values in a finite subset $A$ of an abelian group. Then, letting $p_{\rm max} = \max_{x\in A} \pr(X=x)$, we have
$$p_{\rm max} \leq \frac{\Eta(X+X')-\Eta(X) + 3/2}{\Eta(X)},$$
where $X'$ is an independent copy of $X$.
In particular, if
$$p_{\rm max} \geq \frac{3/2+C}{\Eta(X)},$$
then 
$$\Eta(X+X') \geq \Eta(X) + C.$$
\end{proposition}

\begin{proof}
As before, let
$h_2(q)$ denote the entropy of a random variable that is $1$ 
with probability $q$ and $0$ otherwise,
$q\in [0,1]$.
Let $p = \pr(X\ne 0)$ and assume without loss of generality 
that $p_{\rm max} = 1-p$. Repeating the proof 
of~\cite[Lemma~5.1]{sumsetrevisited}, with $X+X'$ in place of $X-X'$, 
up to the derivation of~\cite[Eq.~(5.9)]{sumsetrevisited} 
with $p_{\rm max} = (1-p)$, yields the bound
\begin{align*}
\Eta(X+X') &\geq (2-p)\Eta(X) + ph_2(p)-2p(1-p)-h_2(p^2)\cr
&\geq (2-p)\Eta(X) - \frac{1}{2} - 1\cr
&=  \Eta(X) +   p_{\rm max}
\Eta(X) - \frac{3}{2},
\end{align*}
where we used that $p(1-p) \leq 1/4$ and $h_2(p^2)\leq 1$,
$p\in[0,1]$. The desired inequality follows. 
\end{proof}

This proposition can be combined with the main result of Section~3 to prove \Cref{thmweaksumproduct}. The statement that both $\Eta(X+X')$ and $\Eta(X X')$ cannot both be very small is not \textit{a~priori} obvious; for example,
if one were to replace $\Eta$ with $\Eta_2$, no such statement holds 
in general, as both additive and multiplicative doublings can simultaneously 
be of constant order; see~\cite[Footnote~18]{biw2006}.

\begin{proof}[Proof of \Cref{thmweaksumproduct}.]
Let $C>0$ and $\eps>0$ be real parameters.
Under the assumptions in the theorem (one of which involves a parameter $K_{C,\eps}$ to be specified later,
our goal is to prove that
$$\Eta(X_1X_2) \ge \biggl( \frac76-\eps \biggr) \Eta(X),$$
so long as
$$\Eta(X_1+X_2) \le \Eta(X) + C$$
and $\Eta(X)$ is large enough depending on $C$ and $\eps$.

If $\eps > 1/6$, then the claim simply follows from the inequality $\Eta(X_1X_2) \geq 
\pr(X\neq 0)\Eta(X)$ combined with \Cref{propmaxprob}. Hence we assume that $0<\eps \leq 1/6$. 

Using the condition on the additive doubling and \Cref{proptaofivetwo}, we can express $X = U+Z$, where $U$ is uniform on a coset progression $H+P$ of cardinality $\Theta_C(2^{\Eta(X)})$ and rank $O_C(1)$, and $\Eta(Z) = O_C(1)$. 

Fixing $C' \geq H(Z)$, we let
\begin{equation} \label{deltadef}
    \delta = 2^{-2C'/\eps}
\end{equation} 
and $A = \bigl\{z\in F :\pr(Z=z)\geq \delta\bigr\}.$
Let $X',Z'$ have the distributions of $X,Z$, respectively,
conditioned on $\{Z\in A\}$, so that
$X' = U+Z'$ is supported on $H+P+A$.
Then
\begin{align*}
\pr(X' = x)  &= \frac{1}{\pr(Z \in A)}\sum_{z \in A }\pr(U = x-z, Z = z) \\
&\leq \frac{1}{\pr(Z \in A)}\sum_{z \in A }\pr(U = x-z) \\
&\le \frac{|A|}{\pr(Z\in A)|H+P|}  \cr
&\le \frac{1} {\Theta_{C,\delta}\bigl(2^{\Eta(X)}\bigr)},
\end{align*}
since $Z'$ is supported on $A$, where $|A|\le 1/\delta$. 
Equivalently, we have shown that
\begin{equation} \label{XprimeHmin}
\Etamin(X') \geq \Eta(X) - O_{\delta, C}(1).
\end{equation}
Now we apply \Cref{thmminentropy} to $X'$ to obtain
\begin{equation}\begin{aligned}
     \max\bigl\{ \Eta(X_1'+X_2'), \Eta( X_1'X_2') \bigr\} &\ge 
\frac{4 \Eta(X') + 3\Etamin(X')}{6} - O(1) \\
&\geq \frac{7}{6}\Etamin(X') - O(1)\\ \label{primefirstsp}
&\geq \frac{7}{6}\Eta(X) - O_{\delta,C}(1),
\end{aligned}\end{equation}
where $X_1',X_2'$ are independent copies of $X'$ and we used \eqref{XprimeHmin} in the last inequality. To verify that $X'$ satisfies the assumption of \Cref{thmminentropy} in the case of positive characteristic, note that since $X'$ takes values in $H+P+A$ with $|A| \leq 1/\delta$ and $|H+P+A|\leq |H+P|/\delta$, we must have $\max_{x' \in H+P+A}\pr(X'=x') \geq \delta/|H+P|$.
Consequently,
$$\Etamin(X') \leq \log{\frac1\delta} + \log{|H+P|} \leq \log{\frac1\delta} + H(X)+H(Z) \leq \Eta(X) +C'+ \log{\frac{1}{\delta}},$$
and the assumption is satisfied by $X'$ provided that $\Eta(X) \leq 2/3\log{p} - K_1-C'-\log(1/\delta)$, where $K_1$ is the constant appearing in \Cref{thmminentropy}. Thus it suffices to take 
\begin{equation} \label{K1eps}
    K_{C,\eps} = C'+\log{\frac{1}{\delta}}+K_1,
\end{equation}
with our earlier choice~\eqref{deltadef} of $\delta$.

By \Cref{lemZsupport} (and possibly taking $\delta$ smaller depending only on $C$) we have
\begin{equation}\begin{aligned} \label{taoprimesestimate}
C \ge \Eta(X_1 + X_2) - \Eta(X)
\ge \bigl(\Eta(X_1' + X_2') - \Eta(X')\bigr)\biggl(1-\frac{2C'}{\log(1/\delta)}\biggr) - C'\frac{\log\log(1/\delta)}{\log(1/\delta)},
\end{aligned}\end{equation}
hence
$$\Eta(X_1' + X_2') \le \Eta(X') + O_{C,\delta}(1).$$
But then, since
$$\Eta(X') \pr(Z\in A) \le \Eta\bigl(X\given 
\one_{\{Z\in A\}}\bigr)\le \Eta(X)$$
and
$$\pr(Z\in A) \ge 1- \frac{C'}{\log(1/\delta)},$$
we have
$$\Eta(X_1'+X_2')\le \biggl(\frac{1}{ 1-C'/\log(1/\delta)}\biggr) \Eta(X) + O_{C,\delta}(1),$$
by~\eqref{eqCoverlog}.
But by our choice of $\delta$ we have $1/\bigl(1-C'/\log(1/\delta)\bigr) <7/6$, so~\eqref{primefirstsp} forces
\begin{equation} \label{multdoublinglarge}
\Eta(X_1'X_2') -\Eta(X') \geq \frac{1}{6}\Eta(X) - O_{\delta,C}(1),
\end{equation}
provided that $\Eta(X)$ (and hence $\Eta(X')$) is large enough depending on $\delta$ and $C$.

Now, by observing that~\eqref{multdoublinglarge} also implies $\Eta(X_1' X_2') - \Eta(X') \geq 0$ for $\Eta(X)$ large enough, we rearrange the multiplicative statement of \Cref{lemZsupport} to
$$\frac{\Eta(X_1 X_2)}{\pr(X\ne 0)} \ge \Eta(X)
+ \bigl(\Eta(X_1' X_2') - \Eta(X')\bigr)\biggl(1-\frac{C'}{\log(1/\delta)}\biggr)^2 - O_{\delta,C}(1),$$
and by \eqref{multdoublinglarge} we now have
$$
\frac{\Eta(X_1 X_2)}{\pr(X\ne 0)} \ge \Eta(X) + \Eta(X)\biggl(\frac16-\frac{C'}{3\log(1/\delta)}\biggr) - O_{\delta,C}(1).
$$
We can make the factor of $1/\pr(X\ne 0)$ disappear by using \Cref{propmaxprob}; under our hypotheses this states that
$$\pr(X\ne 0) \geq 1- \frac{3/2+C}{\Eta(X)}.$$
Hence
$$
\Eta(X_1 X_2) \ge  \Eta(X)\biggl(\frac76-\epsilon\biggr),
$$
by the choice of $\delta$, as long as $\Eta(X)$ is large enough depending on $\delta, \epsilon$ and $C.$
\end{proof}

\section*{Acknowledgements}

We would like to thank Leonardo Franchi and Oliver Roche-Newton for helpful discussions. The first and third named authors were supported in part by the {\small EPSRC}-funded {\small INFORMED-AI} project EP/Y028732/1. The second author is funded by the Natural Sciences and Engineering Research Council of Canada; he also wishes to thank Timothy Gowers for hosting his visit to the University of Cambridge, where the bulk of this paper was written.

\bibliographystyle{alphacitation}
\xpatchcmd{\em}{\itshape}{\slshape}{}{}
\bibliography{citations}

\begin{thebibliography}{AYMRS17}
\newcommand{\enquote}[1]{``#1''}

\bibitem[AYMRS17]{yazicietal2017}
Esen Aksoy~Yazici, Brendan Murphy, Misha Rudnev, and Ilya Shkredov.
\newblock \enquote{Growth estimates in positive characteristic via collisions.}
\newblock \emph{International Mathematics Research Notices} \textbf{23} (2017), 7148--7189.

\bibitem[BIW06]{biw2006}
Boaz Barak, Russell Impagliazzo, and Avi Wigderson.
\newblock \enquote{Extracting randomness using few independent sources.}
\newblock \emph{SIAM Journal on Computing} \textbf{36} (2006), 1095--1118.

\bibitem[BL19]{basitlund2019}
Abdul Basit and Ben Lund.
\newblock \enquote{An improved sum-product bound for quaternions.}
\newblock \emph{SIAM Journal on Discrete Mathematics} \textbf{33} (2019), 1044--1060.

\bibitem[Blo25]{bloom2025}
Thomas Bloom.
\newblock \enquote{Control and its applications in additive combinatorics.}
\newblock \emph{arXiv preprint 2501.29470}  (2025).

\bibitem[Cau13]{cauchy}
Augustin-Louis Cauchy.
\newblock \enquote{Recherches sur les nombres.}
\newblock \emph{Journal de {l'\'{E}cole} Polytechnique} \textbf{9} (1813), 99--116.

\bibitem[CG88]{CG88}
Benny Chor and Oded Goldreich.
\newblock \enquote{Unbiased bits from sources of weak randomness and probabilistic commuication complexity.}
\newblock \emph{SIAM Journal on Computing} \textbf{17} (1988), 230--261.

\bibitem[CT06]{cover:book2}
Thomas~Merrill Cover and Joy~Aloysius Thomas.
\newblock \emph{Elements of Information Theory} (New York: John Wiley and Sons, 2nd edition, 2006).

\bibitem[Dav35]{davenport}
Harold Davenport.
\newblock \enquote{On the addition of residue classes.}
\newblock \emph{Journal of the London Mathematical Society} \textbf{1} (1935), 30--32.

\bibitem[dZ16]{dezeeuw2016}
Frank de~Zeeuw.
\newblock \enquote{A short proof of {R}udnev’s point-plane incidence bound.}
\newblock \emph{arXiv preprint 1612.02719}  (2016).

\bibitem[GGMT24]{martontorsion}
William~Timothy Gowers, Ben Green, Freddie Manners, and Terence Tao.
\newblock \enquote{Marton's conjecture in abelian groups with bounded torsion.}
\newblock \emph{arXiv preprint 2404.02244}  (2024).

\bibitem[GGMT25]{pfr}
William~Timothy Gowers, Ben Green, Freddie Manners, and Terence Tao.
\newblock \enquote{On a conjecture of Marton.}
\newblock \emph{Annals of Mathematics} \textbf{201} (2025), 515--549.

\bibitem[GK24]{GKclt}
Lampros Gavalakis and Ioannis Kontoyiannis.
\newblock \enquote{Entropy and the discrete central limit theorem.}
\newblock \emph{Stochastic Processes and Their Applications} \textbf{170} (2024), 104294.

\bibitem[GMT25]{sumsetrevisited}
Ben Green, Freddie Manners, and Terence Tao.
\newblock \enquote{Sumsets and entropy revisited.}
\newblock \emph{Random Structures \& Algorithms} \textbf{66} (2025), e21252.

\bibitem[Goh24]{goh2024}
Marcel~Kieren Goh.
\newblock \enquote{On an entropic analogue of additive energy.}
\newblock \emph{arXiv preprint 2406.18798}  (2024).

\bibitem[HAT14]{hat}
Saeid Haghighatshoar, Emmanuel Abbe, and Emre Telatar.
\newblock \enquote{A new entropy power inequality for integer-valued random variables.}
\newblock \emph{IEEE Transactions on Information Theory} \textbf{60} (2014), 3787--3796.

\bibitem[KM14]{KM:14}
Ioannis Kontoyiannis and Mokshay Madiman.
\newblock \enquote{Sumset and inverse sumset inequalities for differential entropy and mutual information.}
\newblock \emph{IEEE Transactions on Information Theory} \textbf{60} (2014), 4503--4514.

\bibitem[KMPS20]{kohetal2020}
Doowon Koh, Mirzael Mozhgan, Thang Pham, and Chun-Yen Shen.
\newblock \enquote{Exponential sum estimates over prime fields.}
\newblock \emph{International Journal of Number Theory} \textbf{16} (2020), 291--308.

\bibitem[Kol15]{kollar2015}
J\'anos Koll\'ar.
\newblock \enquote{{S}zemerédi–{T}rotter-type theorems in dimension 3.}
\newblock \emph{Advances in Mathematics} \textbf{271} (2015), 30--61.

\bibitem[LGK26]{lgk2025}
Rupert Li, Lampros Gavalakis, and Ioannis Kontoyiannis.
\newblock \enquote{Entropic additive energy and entropy inequalities for sums and products.}
\newblock \emph{IEEE Transactions on Information Theory} \textbf{72} (2026), 1553--1568.

\bibitem[Li26]{li2026}
Rupert Li.
\newblock \enquote{The entropic sum-product phenomenon.}
\newblock \emph{arXiv preprint 2607.29042}  (2026).

\bibitem[LRN11]{LiRocheNewton}
Liangpan Li and Oliver Roche-Newton.
\newblock \enquote{An improved sum-product estimate for general finite fields.}
\newblock \emph{SIAM Journal on Discrete Mathematics} \textbf{25} (2011), 1285--1296.

\bibitem[MO25]{matheoregan}
Andr\'as M\'ath\'e and William O'Regan.
\newblock \enquote{Discretised sum-product theorems by Shannon-type inequalities.}
\newblock \emph{Journal of the London Mathematical Society} \textbf{112} (2025), e70389.

\bibitem[Poh26]{pohoata2026}
Cosmin Pohoata.
\newblock ``Sum-product in finite fields via entropy.'' Blog post at \texttt{https://pohoatza.wordpress.com/2026/01/12/sum-product-in-finite -fields-via-entropy}, January 12 (2026).

\bibitem[Rag25]{pfrimproved}
Rushil Raghavan.
\newblock \enquote{Improved bounds for the {Freiman-Ruzsa} theorem.}
\newblock \emph{arXiv preprint 2512.11217}  (2025).

\bibitem[RNR15]{roru2015}
Oliver Roche-Newton and Misha Rudnev.
\newblock \enquote{On the {M}inkowski distances and products of sum sets.}
\newblock \emph{Israel Journal of Mathematics} \textbf{209} (2015), 507--526.

\bibitem[Ruz09]{ruzsa2009}
Imre Ruzsa.
\newblock \enquote{Sumsets and entropy.}
\newblock \emph{Random Structures and Algorithms} \textbf{34} (2009), 1--10.

\bibitem[Sha48]{shannon:48}
Claude~Elwood Shannon.
\newblock \enquote{A mathematical theory of communication.}
\newblock \emph{Bell System Technical Journal} \textbf{27} (1948), 379--423, 623--656.

\bibitem[Sta59]{stam:59}
Aart~Johannes Stam.
\newblock \enquote{Some inequalities satisfied by the quantities of information of {F}isher and {S}hannon.}
\newblock \emph{Information and Control} \textbf{2} (1959), 101--112.

\bibitem[Tao10]{tao2010}
Terence Tao.
\newblock \enquote{Sumset and inverse sumset theory for {S}hannon entropy.}
\newblock \emph{Combinatorics, Probability, and Computing} \textbf{19} (2010), 603--639.

\bibitem[TV08]{taovu2008}
Terence Tao and Van~Ha Vu.
\newblock \enquote{John-type theorems for generalized arithmetic progressions and iterated sumsets.}
\newblock \emph{Advances in Mathematics} \textbf{219} (2008), 428--449.

\bibitem[Vad12]{vadhan2012}
Salil~Pravin Vadhan.
\newblock \enquote{Pseudorandomness.}
\newblock \emph{Foundations and Trends in Theoretical Computer Science} \textbf{7} (2012), 1--336.

\end{thebibliography}

\bigskip\noindent
\textsc{Statistical Laboratory, {\small DPMMS}, University of Cambridge, Centre for Mathematical Sciences, Wilberforce Road, Cambridge {\small CB3$\;$0WB}, {\small U.K.}}

\smallskip\noindent
\textsl{E-mail address}: \texttt{lg560@cam.ac.uk}

\bigskip\noindent
\textsc{Department of Mathematics and Statistics, McGill University, Montr\'eal, Qu\'ebec {\small H3A$\;$0B9}, Canada}

\smallskip\noindent
\textsl{E-mail address}: \texttt{marcel.goh@mail.mcgill.ca}

\bigskip\noindent
\textsc{Statistical Laboratory, {\small DPMMS}, University of Cambridge, Centre for Mathematical Sciences, Wilberforce Road, Cambridge {\small CB3$\;$0WB}, {\small U.K.}}

\smallskip\noindent
\textsl{E-mail address}: \texttt{yiannis@maths.cam.ac.uk}

\end{document}